\documentclass[11pt,leqno]{amsart} 

\usepackage{amsmath} 
\usepackage{amsthm}  
\usepackage[english,french]{babel}

\begin{document}
\newtheorem{Lemma}{\bf LEMMA}[section]
\newtheorem{Theorem}[Lemma]{\bf THEOREM}
\newtheorem{Claim}[Lemma]{\bf CLAIM}
\newtheorem{Corollary}[Lemma]{\bf COROLLARY}
\newtheorem{Proposition}[Lemma]{\bf PROPOSITION}
\newtheorem{Example}[Lemma]{\bf EXAMPLE}
\newtheorem{Fact}[Lemma]{\bf FACT}
\newtheorem{Definition}[Lemma]{\bf DEFINITION}
\newcommand{\restrict}{\mbox{$\mid\hspace{-1.1mm}\grave{}$}}
\newcommand{\covers}{\mbox{$>\hspace{-2.0mm}-{}$}}
\newcommand{\covered}{\mbox{$-\hspace{-2.0mm}<{}$}}
\newcommand{\notcover}{\mbox{$>\hspace{-2.0mm}\not -{}$}}

\newcommand{\Q}{\mathbb{Q}}
\newcommand{\Z}{\mathbb{Z}}
\newcommand{\N}{\mathbb{N}}
\newcommand{\disju}{\dot \cup}  
\newcommand{\Hom}{\textrm{Hom}}
\newcommand{\End}{\textrm{End}}
\newcommand{\Min}{\textrm{Min}}
\newcommand{\Max}{\textrm{Max}}
\newcommand{\Aut}{\textrm{Aut}}
\newcommand{\id}{\textrm{id}}

\title{Representable Posets and their Order Components}
\author{M.~E.~Adams and D.~van der Zypen }
\date{}
\maketitle
\selectlanguage{french}
\begin{abstract}
Un ensemble partiellement ordonn\'e $P$ est {\it repr\'esentable} si il existe un 
$(0,1)$-tr\'eillis distributif, dont l'ensemble ordonn\'e des id\'eaux primes
est isomorphe \`a $P$. Dans cet article, nous voulons d\'emontrer 
que si tous les composants
d'ordre de $P$ sont repr\'esentables, $P$ est repr\'esentable aussi. En plus,
nous montrons que, bien que la 
topologie d' intervalle de chaque composant soit
compacte, il existe un ensemble partiellement ordonn\'e  
qui est repr\'esentable et qui poss\`ede un composant d'ordre
non-repr\'esentable.
\end{abstract}

\selectlanguage{english}

\section{Introduction} \label{A}
A poset is said to be {\it representable} if it is
isomorphic to the poset of prime ideals of a bounded distributive
lattice (that is a distributive lattice with a largest element $1$
and a smallest element $0$).  The question of which posets are
representable essentially dates back to Balbes \cite{Ba71} 
(see also, Balbes and Dwinger \cite{BaDw74}) and has been
considered by a number of authors since (see, for example, the
expository article Priestley
\cite{Pr94}.)\\

In \cite{Pr70}, Priestley proved that the category $\mathcal{D}$
of bounded distributive lattices with $(0,1)$-preserving lattice
homomorphisms and the category $\mathcal{P}$ of compact totally
order-disconnected spaces (henceforth referred to as {\it
Priestley spaces}) with order-preserving continuous maps are
dually equivalent. (A compact {\it totally order-disconnected
space} $(X;\tau ,\leq )$ is a poset $(X;\leq )$ endowed with a
compact topology $\tau$ such that, for $x$, $y\in X$, whenever
$x\not\geq y$, then there exists a clopen decreasing set $U$ such
that $x\in U$ and $y\not\in U$.) The functor $D:{\mathcal{D}}\to
{\mathcal{P}}$ assigns to each object $L$ of ${\mathcal{D}}$ a
Priestley space $(D(L);\tau (L),\subseteq)$, where $D(L)$ is the
set of all prime ideals of $L$ and $\tau (L)$ is a suitably
defined topology (the details of which will not be required here).
The functor $E:{\mathcal{P}}\to {\mathcal{D}}$ assigns to each
Priestley space $X$ the lattice $(E(X);\cup ,\cap ,\emptyset ,X)$, where $E(X)$
is the set of all clopen decreasing sets of $X$.  In particular, a
poset $(X;\leq )$ is seen to be representable iff there exists a
topology $\tau$ such that $(X;\tau ,\leq )$ is a Priestley
space.\\

Let $(X;\leq)$ be a poset. Then we define a relation $R$ on $X$ by
setting $(x,y) \in R$ whenever $x \leq y$ or $y \leq x$. Let $R'$
be the transitive closure of $R$. Then $R'$ is an equivalence
relation. An \textsl{order component} of $X$ is an equivalence
class $[x]_{R'}$ of the relation $R'$ for some $x \in X$. Further,
for any $Y\subseteq X$, let $(Y] = \{ x\in X\mid x\leq y \mbox{
for some } y\in Y\}$ and $[Y) = \{ x\in X\mid x\geq y \mbox{ for
some } y\in Y\}$. Should $Y = \{ y\}$ for some $y\in X$, then, for
simplicity, we will denote $(Y]$ and $[Y)$ by $(y]$ and $[y)$,
respectively. Finally, let $[x,y] = [x) \cap (y]$,
${\mathcal{S}}^- =\{X\setminus (x]\mid x \in X\}$, and
${\mathcal{S}}^+ =\{X\setminus [x)\mid x \in X\}$.  Then
${\mathcal{S}} ={\mathcal{S}}^- \cup {\mathcal{S}}^+$ is an open
subbase for the so called {\it interval topology} $\tau_i$ on $X$
(sometimes, in the interest of clarity, $\tau_i$ will be denoted
$\tau_i(X)$ when we wish to emphasize the poset concerned). It is
well known that if $(X; \tau, \leq)$ is a Priestley space, then
$\tau$ contains the interval
topology $\tau_i$.\\

Our principal result is the following:\\

\begin{Theorem} \label{A1}
If the order components of a poset $(X;\leq )$ are
representable, then so is $X$. However, even though each order
component of a representable poset is compact under its interval
topology, there exists a representable poset with an order
component which is not representable.
\end{Theorem}
\ \\

The proof of \ref{A1} will be given in \S \ref{B}, where we begin in
\ref{B3} by showing that a poset is compact under its interval
topology iff each order component is compact under its respective
interval topology. As observed in
\ref{B4},
it follows readily from this that each
order component of a representable poset is compact with respect
to its interval topology.  We then establish in \ref{B5}
that if every order component of a poset is representable, then so
too is the poset.  Finally, we define a countably infinite poset
which we show to be order-isomorphic to an order component of a
representable poset in \ref{B2}, but which, as we show in \ref{B6}, 
is not itself representable.\\

For any undefined terms or additional background, we refer the
reader to the texts Gr\"{a}tzer \cite{Gr98} and Kelley
\cite{Ke55}, with each of which our notation is consistent.

\section{Proof of ${\bf 1.1}$} \label{B}

\begin{Lemma} \label{B3}
Let $(X_k; \leq_k)_{k \in K}$ be a family of
pairwise disjoint nonempty posets. Then for $(X;\leq )$ where $X=
\bigcup_{k \in K}X_k$ and $\leq = \bigcup_{k \in K} \leq_k$, the
following are equivalent{\rm :}\\

{\rm (}i{\rm )} for each $k \in K$, the space $(X_k; \tau_i(X_k))$
is
compact{\rm ;}\\

{\rm (}ii{\rm )}  $(X; \tau_i(X))$ is compact.\\
\label{union_of_interval_compact}
\end{Lemma}
\begin{proof}
Assume that (i) holds and let $\mathcal{U}$ be an open cover of $X
=\bigcup_{k \in K}X_k$. By Alexander's subbase lemma, we may
assume that $${\mathcal{U}}=\{X \backslash (a] \mid a \in A\} \cup
\{X \backslash [b) \mid b \in B\}$$ for some subsets
$A , B \subseteq X$. We distinguish two cases: \\

First, there is some $k \in K$ such that $A \cup B \subseteq X_k$.
In which case, consider ${\mathcal{U}}_{X_k} = \{X_k \backslash (a]
\mid a \in A\} \cup \{X_k \backslash  [b) \mid b \in B\}$. Since
$(X_k; \tau_i(X_k))$ is compact by assumption,
${\mathcal{U}}_{X_k}$ has a finite subcover
$$\{X_k\backslash (a_1], ..., X_k\backslash
(a_r]\} \cup \{X_k\backslash  [b_1), ..., X_k\backslash [b_s)\},$$
so $\{X\backslash  (a_1], ..., X\backslash  (a_r]\} \cup
\{X\backslash  [b_1), ..., X\backslash [b_s)\}$ is a finite
subcover of $\mathcal{U}$. Second, there is no $k \in K$ such that
$A \cup B \subseteq X_k$. In which case there are $w_1, w_2 \in A
\cup B$ such that $w_1 \in X_{k}$ and $w_2\in X_{k'}$ for some $k
\neq k' \in K$. If $w_1, w_2 \in A$, then $\{X\backslash (w_1], X
\backslash (w_2]\}$ is a finite subcover of $\mathcal{U}$. If $w_1
\in A, w_2 \in B$, then $\{X\backslash (w_1], X \backslash
[w_2)\}$ is a finite subcover of $\mathcal{U}$ (similarly for $w_1
\in B, w_2 \in A$). Finally if $w_1, w_2 \in B$, then
$\{X\backslash [w_1), X \backslash [w_2)\}$ is a
finite subcover of $\mathcal{U}$.\\

Thus, in any case, $(X;\tau_i(X))$ is compact.\\

Assume that (ii) holds and let $k\in K$. Assume that $\mathcal{U}$
is an open cover of $X_k$. By Alexander's subbase lemma we may
assume that $${\mathcal{U}}=\{X_k \backslash  (a] \mid a \in A\}
\cup \{X_k \backslash [b) \mid b \in B\}$$ for some subsets $A , B
\subseteq X_k$. Consider the following open cover of $X=\bigcup_{l
\in K}X_l$
$${\mathcal{U}}^* =\{X \backslash (a] \mid a \in A\} \cup
\{X \backslash [b) \mid b \in B\}. $$ Then ${\mathcal{U}}^*$ has a
finite subcover $\{X\backslash (a_1], ..., X\backslash (a_r]\}
\cup \{X\backslash [b_1), ..., X\backslash [b_s)\}$ since $X$ is
compact with its interval topology. Thus $\{X_k\backslash (a_1],
..., X_k\backslash (a_r]\} \cup \{X_k\backslash [b_1), ...,
X_k\backslash [b_s)\}$ is a finite subcover of $X_k$.
\end{proof}

If $(X;\leq )$ is representable, then, for some topology $\tau$,
$(X;\tau ,\leq )$ is a Priestley space.  In particular, $(X;\tau
)$ is a compact space and, as $\tau_i\subseteq \tau$, so too is
$(X;\tau_i)$.  Thus, the following is an immediate consequence of
\ref{B3}.
\begin{Lemma} \label{B4}
Each order component of a representable poset is
compact with respect to its interval topology.
\end{Lemma}

We now go on to show that if the order components of a poset are
representable, then so is the poset.
\begin{Lemma} \label{B5}
Let $(X_k, \leq_k)_{k \in K}$ be a family of
pairwise disjoint nonempty representable posets. Then $(X;\leq )$
is representable, where $X= \bigcup_{k \in K}X_k$ and $\leq =
\bigcup_{k \in K} \leq_k$.
\end{Lemma}
\begin{proof}
If $K$ is empty or a singleton, the statement is trivial. So we
may assume that $K$ has more than one element. For any $k \in K$,
let $\tau_k$ be a topology making $(X_k;\tau_k, \leq_k)$ a
Priestley space. Fix $k \in K$ and $x \in X_k$. We now build
a subbase for a topology on $X$ in three steps. We set:\\

${\mathcal{S}}_1= \bigcup_{l \in K\backslash\{k\}} \tau_l$;\\

${\mathcal{S}}_2= \{U \in \tau_{k} \mid x \notin U\}$;\\

${\mathcal{S}}_3= \{U \subseteq X \mid x \in U \textrm{ and }U
\cap X_{k} \in \tau_{k} \textrm{ and, for some } k' \in
K\backslash\{k\},\ U=[U \cap X_{k}] \cup [\bigcup_{l \in
K\backslash\{k,k'\}} X_l]\}$.\\

Then let $\tau$ be the topology having
${\mathcal{S}}={\mathcal{S}}_1 \cup {\mathcal{S}}_2
\cup{\mathcal{S}}_3$ as a subbase. Using Alexander's subbase lemma
we check easily that $(X;\tau)$ is compact using the fact that any
subbase member containing $x$ is, in some sense, large by virtue of 
the definition
of ${\mathcal{S}}_3 \subseteq {\mathcal{S}}$. Moreover, an easy
distinction by cases tells us that $(X;\tau ,\leq )$ is totally
order-disconnected.
\end{proof}

It remains to provide an example of a poset $(P;\leq )$ which is
order isomorphic to an order component of a representable poset,
but is
not representable itself.\\

On the set
\[ P = \{p\}\cup \{ p_{i_0,\ldots ,i_n}\mid 0\leq n<\omega \mbox{ and } 0\leq i_j<\omega \mbox{ for }0\leq j\leq n\}, \]
inductively define an order relation $\leq$ as follows.\\

For $0\leq j<i<\omega$,
\[ p<p_i<p_j. \]

For $0\leq i_0<\omega$, $0\leq k\leq i_0$, and $0\leq i<j<\omega$,
\[ p_{i_0,i}<p_{i_0,j}<p_k. \]

For $0\leq i_0,i_1<\omega$, $0\leq k\leq i_1$, and $0\leq
j<i<\omega$,
\[ p_{i_0,k}<p_{i_0,i_1,i}<p_{i_0,i_1,j}.\]

In general, let $0<r<\omega$.\\

For $0\leq i_0,i_1,\ldots ,i_{2r}<\omega$, $0\leq k\leq i_{2r}$,
and $0\leq i<j<\omega$,
\[  p_{i_0,i_1, \ldots ,i_{2r-1},i_{2r},i}  < p_{i_0,i_1, \ldots ,i_{2r-1},i_{2r},j}  <p_{i_0, i_1,\ldots
,i_{2r-1},k}.\]

For $0\leq i_0,i_1,\ldots ,i_{2r+1}<\omega$, $0\leq k\leq
i_{2r+1}$, and $0\leq j<i<\omega$,

\[
p_{i_0, i_1,\ldots ,i_{2r},k}<p_{i_0,i_1, \ldots
,i_{2r},i_{2r+1},i} <p_{i_0,i_1, \ldots ,i_{2r},i_{2r+1},j}. \]  \\

To see that $(P;\leq )$ is a poset, for $0\leq n<\omega$, let
\[  P(n) = \{ p\}\cup \{ p_{i_0,\ldots ,i_m}\mid \mbox{ $0\leq m\leq n$ and, for $0\leq j\leq m$, $0\leq
i_j<\omega \}$}.  \]
Thus, $P(0)$ and, for each $0\leq n<\omega$,
$P(n+1)\setminus P(n)$ are clearly antisymmetric and transitive.
Further, $x\in P(n)$ is comparable with $y\in P\setminus P(n)$
only if $x\in P(n)\setminus P(n-1)$ and $y\in P(n+1)\setminus
P(n)$, where it is the case that $x>y$ and $x<y$ depending on
whether $n$ is even or odd, respectively.  In particular,  $\leq$
is antisymmetric. Moreover, if $n$ is even, say $n = 2r$, then $x
= p_{i_0, \ldots , i_{2r-1},k}$ and $y = p_{i_0, \ldots
,i_{2r},i}$ providing $0\leq k\leq i_{2r}$ and $0\leq i<\omega$,
and if $n$ is odd, say $n = 2r+1$, then $x = p_{i_0,\ldots ,
i_{2r},k}$ and $y = p_{i_0,\ldots ,i_{2r+1},i}$ providing $0\leq
k\leq i_{2r+1}$ and $0\leq i<\omega$. In particular, $\leq$ is
transitive and, as claimed, $(P;\leq )$ is seen to be a countable
connected poset. We also note in passing that, for $0\leq i_0,
\ldots ,i_n<\omega$, $[p_{i_0, \ldots ,i_n})$ and $(p_{i_0, \ldots
,i_n}]$ are finite chains depending on whether $n$ is even or odd,
respectively, a fact that we will refer back to later.\\

In order to show that $(P;\leq )$ is order-isomorphic to an order
component of a representable poset, we will define a suitable
order $\preceq$ on a compact totally disconnected space $(C;\tau
)$ which itself is homeomorphic to the {\it Stone space} of a
countable atomless Boolean algebra.  To do so, we will need an
explicit description of $(C;\tau )$, which we now give.\\

Let ${\bf Q} = (Q;\leq)$ denote the rational interval $(0,1)$.
Then $(A,B)$ is a \textit{Dedekind cut} of $Q$ providing that $A$
and $B$ are disjoint non-empty sets such that $Q = A\cup B$ and,
for $a\in A$ and $b\in B$, $a<b$.  For a Dedekind cut $(A,B)$ of
${\bf Q}$, $A$ is a \textit{gap} providing $A$ does not have a
greatest element and $B$ does not have a smallest element and,
otherwise, it is a \textit{jump}.  Let $(C;\leq )$ denote the set
of all decreasing subsets of the rational interval $(0,1)$ ordered
by inclusion.  Thus, for $I\in C$, if $I\neq \emptyset$ or $Q$,
then $I$ is a jump precisely when $I = (0,r)$ or $(0,r]$ for some
$r\in Q$.  Intuitively, $(C;\leq )$ may be thought of as the real
interval $[0,1]$ where every rational element $0<r<1$ is replaced
by a covering pair. The interval topology $\tau_i$, denoted
henceforth simply by $\tau$, on $(C;\leq )$ has as a base the open
intervals $C$, $[\emptyset , I) = \{ J\in C:\ J\subset I\}$,
$(I,Q] = \{ J\in C:\ I\subset J\}$, and $(I,J) = \{K\in C:\
I\subset K\subset J\}$. It is well-known that $(C;\tau )$ is a
compact totally disconnected space, whose clopen subsets are
precisely the sets $\emptyset$, $C$, and finite unions of sets of
the form $[I,J] = \{ K\in C:\ I\subseteq K\subseteq J\}$ where $I
= (0,r]$ and $J = (0,s)$ for $r$, $s\in Q$ with $r<s$.\\

Setting $Q = ( s_i:\ 0\leq i<\omega )$ to be some enumeration of
$Q$, we now inductively define a new partial order on $C$ as
follows:\\

In $C$, choose gaps $x$ and, for $0\leq i<\omega$, $x_i$ such that
\[ x<x_i<x_j \mbox{ for }0\leq j<i<\omega ,  \]
where $x$ is a member of the closure of $\{x_i\mid 0\leq i<\omega\}$, 
denoted $cl(\{x_i\mid 0\leq i<\omega\})$, and set
\[ x\prec x_i\prec x_j. \]
Choose clopen intervals $(X_i:\ 0\leq i<\omega )$
such that $x_i\in X_i$, $X_i\cap X_j = \emptyset$ whenever $i\neq
j$, the length of $X_i$, denoted $ln(X_i)$, is $\leq \frac{1}{2}$ 
in the pseudometric obtained from the
metric imposed on $C$ by the real metric on $(0,1)$, and $(0,s_0)$,
$(0,s_0]\not\in X_i$ for any $0\leq i<\omega$.\\

For $0\leq i_0<\omega$, $0\leq k\leq i_0$, and $0\leq i<\omega$, choose gaps
$x_{i_0,i}\in X_{i_0}$ such that
\[  x_{i_0,i}<x_{i_0,j}< x_k \mbox{ for }0\leq i<j<\omega , \]
where $x_{i_0}\in cl(\{ x_{i_0,i}\mid  0\leq i<\omega
\} )$, and set
\[  x_{i_0,i}\prec x_{i_0,j}\prec x_k. \]
Choose clopen intervals $(X_{i_0,i}:\ 0\leq i<\omega
)$ such that $x_{i_0,i}\in X_{i_0,i}$, $X_{i_0,i}\cap X_{i_0,j} =
\emptyset$ for $i\neq j$, $X_{i_0,i}\subseteq X_{i_0}$,
$ln(X_{i_0,i})\leq \frac{1}{2^2}$, and $(0,s_1)$, $(0,s_1]\not\in
X_{i_0,i}$ for $0\leq i<\omega$.\\

For $0\leq i_0,i_1<\omega$, $0\leq k\leq i_1$,  and $0\leq i<\omega$, choose gaps
$x_{i_0,i_1,i}\in X_{i_0,i_1}$ such that
\[ x_{i_0,k} <  x_{i_0,i_1,i}  < x_{i_0,i_1,j} \mbox{ for }0\leq j<i<\omega , \]
where $x_{i_0,i_1}\in cl (\{ x_{i_0,i_1,i}\mid
0\leq i<\omega \} )$, and set
\[ x_{i_0,k} \prec  x_{i_0,i_1,i}  \prec x_{i_0,i_1,j}. \]
Choose clopen intervals $(X_{i_0,i_1,i}:\ 0\leq
i<\omega )$ such that $x_{i_0,i_1,i}\in X_{i_0,i_1,i}$,
$X_{i_0,i_1,i}\cap X_{i_0,i_1,j} = \emptyset$ for $i\neq j$,
$X_{i_0,i_1,i}\subseteq X_{i_0,i_1}$, $ln(X_{i_0,i_1,i})\leq
\frac{1}{2^3}$, and $(0,s_2)$, $(0,s_2]\not\in X_{i_0,i_1,i}$ for
$0\leq i<\omega$.\\

In general, let $0<r<\omega$.\\

For $0\leq i_0,i_1,\ldots ,i_{2r}<\omega$, $0\leq k\leq i_{2r}$, and $0\leq
i<\omega$, choose gaps $x_{i_0,i_1,\ldots ,i_{2r},i}\in
X_{i_0,i_1,\ldots ,i_{2r}}$ such that
\[  x_{i_0,i_1, \ldots i_{2r -1},i_{2r},i} < x_{i_0,i_1, 
\ldots i_{2r -1},i_{2r},j} < x_{i_0, i_1,\ldots ,i_{2r-1},k} \mbox{ for }0\leq j<i<\omega , \]
where $x_{i_0,\ldots ,i_{2r}}\in cl(\{
x_{i_0,\ldots ,i_{2r},i}\mid 0\leq i<\omega \})$, and set
\[  x_{i_0,i_1, \ldots i_{2r -1},i_{2r},i} \prec x_{i_0,i_1, 
\ldots i_{2r -1},i_{2r},j} \prec x_{i_0, i_1,\ldots ,i_{2r-1},k}. \]
Choose clopen intervals $(X_{i_0, i_1,\ldots ,i_{2r},i}:\
0\leq i<\omega )$ such that $x_{i_0, i_1,\ldots ,i_{2r},i}\in
X_{i_0, i_1,\ldots ,i_{2r},i}$, $X_{i_0, i_1,\ldots ,i_{2r},i}
\cap X_{i_0, i_1,\ldots ,i_{2r},j} = \emptyset$ for $i\neq j$, $X_{i_0,
i_1,\ldots ,i_{2r},i} \subseteq X_{i_0, i_1,\ldots ,i_{2r}}$,
$ln(X_{i_0, i_1,\ldots ,i_{2r},i}) \leq \frac{1}{2^{2r+1}}$, and
$(0,s_{2r+1})$, $(0,s_{2r+1}]\not\in X_{i_0, i_1,\ldots
,i_{2r},i}$ for $0\leq i<\omega$.\\

For $0\leq i_0,i_1,\ldots ,i_{2r+1}<\omega$, $0\leq k\leq i_{2r+1}$, and $0\leq
i<\omega$, choose gaps $x_{i_0,i_1,\ldots
,i_{2r+1},i}\in X_{i_0,i_1,\ldots ,i_{2r+1}}$ such that
\[ x_{i_0, i_1,\ldots ,i_{2r},k} < x_{i_0,i_1, \ldots i_{2r},i_{2r+1},i} 
< x_{i_0,i_1, \ldots i_{2r},i_{2r+1},j}  \mbox{  for }0\leq i<j<\omega ,  \]
where $x_{i_0,\ldots ,i_{2r+1}}\in cl(\{
x_{i_0,\ldots ,i_{2r+1},i}\mid 0\leq i<\omega \})$, and set
\[ x_{i_0, i_1,\ldots ,i_{2r},k} \prec x_{i_0,i_1, \ldots i_{2r},i_{2r+1},i} 
\prec x_{i_0,i_1, \ldots i_{2r},i_{2r+1},j}.  \]
Choose clopen intervals $(X_{i_0, i_1,\ldots
,i_{2r+1},i}:\ 0\leq i<\omega )$ such that $x_{i_0, i_1,\ldots
,i_{2r+1},i}\in X_{i_0, i_1,\ldots ,i_{2r+1},i}$, $X_{i_0,
i_1,\ldots ,i_{2r+1},i} \cap X_{i_0, i_1,\ldots ,i_{2r+1},j} = \emptyset$ for
$i\neq j$, $X_{i_0, i_1,\ldots ,i_{2r+1},i} \subseteq X_{i_0,
i_1,\ldots ,i_{2r+1}}$, $ln(X_{i_0, i_1,\ldots ,i_{2r+1},i})
\leq \frac{1}{2^{2(r+1)}}$, and $(0,s_{2r+2})$, $(0,s_{2r+2}]\not\in X_{i_0,
i_1,\ldots ,i_{2r+1},i}$ for $0\leq i<\omega$.\\

Elsewhere on $C$, let $\preceq$ be trivial.  Thus, since $(X;\preceq )$ is 
order-isomorphic to $(P;\leq )$, $(C;\preceq )$ is a poset whose 
order components consist precisely of $X = \{ x \} \cup 
\{ x_{i_0,\ldots ,i_n}\mid 0\leq n < \omega \mbox{
and } 0\leq i_j < \omega \mbox{ for } 0\leq j\leq n \}$ and
$2^{\omega}$ singletons.\\

\begin{Lemma}\label{B1}
$(C;\tau ,\preceq )$ is a Priestley space.
\end{Lemma}
\begin{proof}
As $(C;\preceq )$ is a poset and $(C;\tau )$ is a compact
totally disconnected space, it remains to show that, for $u,v\in
C$, whenever $u\not \succeq v$ there exists a clopen decreasing
set $U$ such that $u\in U$ and $v\not\in U$.\\

Since $(X;\preceq )$ is order-isomorphic to $(P;\leq
)$, we set, for $0\leq n<\omega$, 
\[ X(n) = \{ x\}\cup \{ x_{i_0,\ldots , i_m}\mid 0\leq m\leq n 
\mbox{ and, for } 0\leq j\leq n,\ 0\leq i_j<\omega \} \] 
and observe that, as $[x_{i_0,\ldots , i_n})$ or $(x_{i_0,\ldots ,i_n}]$ 
is a finite chain depending on whether $n$ is even or odd, respectively,  
it follows from the choice of elements in
$X\setminus X(n)$ that, for $0\leq i_0,\ldots , i_n<\omega $,
$\bigcup (X_{i_0,\ldots ,i_{n-1},k}:\ 0\leq k\leq i_n)$ is clopen
increasing or decreasing, accordingly.\\

Consider $u,v\in C$ with $u\not\succeq v$.  In each case we will exhibit 
a clopen decreasing set $U$ such that $u\in U$ and
$v\not\in U$.\\

If $u<v$, then $u\leq (0,s]<v$ for some $s\in Q$.  Since $\preceq$
is compatible with $\leq$, set $U = (\emptyset ,(0,s]]$.  Henceforth, 
we assume that $u>v$ and, in particular, $u$ and $v$
are incomparable under $\preceq$.\\

Suppose there is an infinite sequence $(i_k:\ 0\leq k<\omega )$
such that $u\in X_{i_0,\ldots ,i_k}$ for any $0\leq k<\omega$.
Then, by choice, $u$ is a gap and, since $ln(X_{i_0,\ldots
,i_n})\leq \frac{1}{2^n}$, $v\not\in X_{i_0,\ldots ,i_n}$ for some $0\leq
n<\omega$. Without loss of generality, we may assume that 
$n$ is even.  Set $U = \bigcup (X_{i_0,\ldots
,i_n,l}:\ 0\leq l\leq i_{n+1})$. By
the above observation, $U$ is clopen decreasing,
$u\in U$, and, since $U\subseteq X_{i_0,\ldots ,i_n}$, $v\not\in
U$.\\

Likewise, if there is an infinite sequence $(j_k:\ 0\leq k<\omega
)$ such that $v\in X_{j_0,\ldots ,j_k}$ for any $0\leq k<\omega$,
then $v$ is a gap and, since $ln (X_{j_o,\ldots
,j_m})\leq \frac{1}{2^m}$, $u\not\in X_{j_0,\ldots ,j_m}$ for some $0\leq
m<\omega$.  We may assume, again with no loss in generality, that 
$m$ is
odd.  Set $U = C\setminus \bigcup (X_{j_0,\ldots ,j_m,l}:\ 0\leq
l\leq j_{m+1})$.  Then, $U$ is clopen decreasing,
$v\not\in U$, and, since $U\subseteq C\setminus X_{j_0,\ldots
,j_m}$, $u\in U$.\\

Suppose, for some finite sequence $(i_k:\ 0\leq k\leq n)$, $u\in
X_{i_0,\ldots ,i_n}$, but $u\not\in X_{i_0,\ldots ,i_n,l}$ for any
$0\leq l<\omega$.  Then, providing $u\neq x_{i_0,\ldots ,i_n}$,
it is not hard to see that there exists a clopen set $U$ such that $u\in U$,
$v\not\in U$, and each element of $U$ is incomparable under $\preceq$ to any other element of
$(C;\preceq )$, whereby $U$ is decreasing.  Were it the case
that $u\not\in X_l$ for any $0\leq l<\omega$, then a 
similar set may be defined unless $u = x$.\\

Likewise, suppose it is the case that, for some finite sequence $(j_k:\ 0\leq k\leq
m)$, $v\in X_{j_0,\ldots ,j_m}$, but that $v\not\in X_{j_0,\ldots
,j_m,l}$ for any $0\leq l<\omega$.  Then, providing $v\neq
x_{j_0,\ldots ,j_m}$, there exists a clopen set $V$ such that
$v\in V$, $u\not\in V$, and each element of $V$ is 
incomparable under $\preceq$ to any other element of
$(C;\preceq )$.  In this case, set $U = C\setminus V$.  
Likewise, unless $v = x$, a similar set may be defined whenever
$v\not\in X_l$ for any $0\leq l<\omega$.\\

Thus, it now remains to consider the eventuality that $u = x$ or
$x_{i_0,\ldots ,i_n}$ for some $(i_k:\ 0\leq k\leq n)$ and $v = x$
or $x_{j_0,\ldots ,j_m}$ for some $(j_k:\ 0\leq k\leq m)$.
Observe that, by hypothesis, since $v<u$, $u = x$ is impossible
and, hence, we need only consider $u = x_{i_0,\ldots ,i_n}$ 
for some $(i_k:\ 0\leq k\leq
n)$.  Further,
if $v = x$, then, by hypothesis, $u = x_{i_0,\ldots ,i_n}$ for
some $n>0$.  Since $v\not\in X_{i_0}$ and $u\neq x_{i_0}$, $u\in U
= \bigcup (X_{i_0,k}:\ 0\leq k\leq i_1)\subseteq X_{i_0}$, which,
as observed above, is clopen decreasing.  Thus, in addition, we may
assume that $v = x_{j_0,\ldots ,j_m}$ for some $(j_k:\ 0\leq k\leq
m)$.\\

A number of possibilities still remain to be considered.\\

Suppose first that $n\leq m$.\\

Consider $i_k = j_k$ for all $0\leq k\leq n$.  Then, by hypothesis,
$m\geq n+2$ and, since $u>v$, $n$ is even.  Thus, $V = \bigcup
(X_{i_0,\ldots ,i_n,j_{n+1},l}:\ 0\leq l\leq j_{n+2})$ is clopen
increasing $v\in V$, and $u\not\in V$.  Set $U = C\setminus
V$.\\

Suppose $i_k = j_k$ for all $0\leq k<n$, but $i_n\neq j_n$.  Then,
by hypothesis, $m\geq n+1$. Suppose $n$ is even.  Were it the case
that $i_n>j_n$, then it would follow that $u<v$, contrary to
hypothesis.  Thus, we may assume that $i_n<j_n$.  But then it
follows that $m\geq n+2$.  Thus, $v\in V = \bigcup (X_{i_0,\ldots
,i_{n-1},j_n,j_{n+1},l}:\ 0\leq l\leq j_{n+2})$ which is clopen
increasing and, since $V\subseteq X_{i_0,\ldots ,i_{n-1},j_n}$,
$u\not\in V$.  Suppose $n$
is odd.  Thus, $v\in V = \bigcup (X_{i_0,\ldots ,i_{n-1},j_n,l}:\
0\leq l\leq j_{n+1})$, which is clopen increasing, and again, since $V\subseteq
X_{i_0,\ldots ,i_{n-1},j_n}$, $u\not\in V$. In either case, set $U = C\setminus
V$.\\

Consider, for some $0\leq k\leq n-1$, $i_l = j_l$ for all $0\leq l
<k$, but $i_k\neq j_k$.  If $k$ is even, then $u\in U = \bigcup
(X_{i_0,\ldots ,i_{k-1},i_k,l}:\ 0\leq l\leq i_{k+1})$ which is clopen
decreasing and, since $U\subseteq X_{i_0,\ldots ,i_{k-1},i_k}$ and
$v\in X_{i_0,\ldots ,i_{k-1},j_k}$, $v\not\in U$.  If $k$ is odd,
then $v\in V = \bigcup (X_{i_0,\ldots ,i_{k-1},j_k,l}:\ 0\leq
l\leq j_{k+1})$ which is clopen increasing and, since $V\subseteq
X_{i_0,\ldots ,i_{k-1},j_k}$ and $u\not\in X_{i_0,\ldots
,i_{k-1},j_k}$, $u\not\in V$.  In this case, set $U = C\setminus V$.\\

It remains to consider $n>m$.\\

Suppose $i_k = j_k$ for all $0\leq k\leq m$.  Then, by hypothesis,
$n\geq m+2$ and, since $u>v$, $m$ is odd. Hence, $u\in U = \bigcup
(X_{j_0,\ldots ,j_m,i_{m+1},l}:\ 0\leq l\leq i_{m+2})$ which is clopen
decreasing, whilst $v\not\in U$.\\

Consider $i_k = j_k$ for all $0\leq k<m$, but $i_m\neq j_m$.  By
hypothesis, $n\geq m+1$. Suppose $m$ is even.  Then, 
$u\in U = \bigcup (X_{j_0,\ldots ,j_{m-1},i_m,l}:\ 0\leq
l\leq i_{m+1})$ which is clopen decreasing, and, since $U\subseteq
X_{j_0,\ldots ,j_{m-1},i_m}$, $v\not\in U$.  Suppose $m$ is odd.
Were $i_m<j_m$, then it would follow that $u<v$, contrary to hypothesis.
Thus, we may assume that $i_m>j_m$ and, so, $n\geq m+2$.  Hence, $u\in U = \bigcup
(X_{j_0,\ldots ,j_{m-1},i_m,i_{m+1},l}:\ 0\leq l\leq i_{m+2})$
which is clopen decreasing, and, since it is also the case that
$U\subseteq X_{j_0,\ldots ,j_{m-1},i_m}$, $v\not\in U$.\\

Finally, it remains to consider the case that, for some $0\leq
k\leq m-1$, $i_l = j_l$ for all $0\leq l<k$, but $i_k \neq j_k$.
However, the same argument holds, word for word, as given in the 
analogous case when
$n\leq m$.
\end{proof}

Since the order components of $(C;\tau ,\preceq )$ consist of precisely
$X = \{ x \} \cup \{ x_{i_0,\ldots ,i_n}\mid 0\leq n < \omega \mbox{
and } 0\leq i_j < \omega \mbox{ for } 0\leq j\leq n \}$ and
$2^{\omega}$ singletons and, by choice, $(X;\preceq )$ is
order-isomorphic to $(P;\leq )$, the following is an immediate
consequence of \ref{B1}.\\

\begin{Lemma} \label{B2}
$(P;\leq )$ is order-isomorphic to an order component of a
representable poset.
\end{Lemma}

The proof of \ref{A1} will be complete once we have established the following.\\

\begin{Lemma} \label{B6}
$(P;\leq )$ is not representable.
\end{Lemma}
\begin{proof}

Suppose, contrary to hypothesis, that $(P;\leq )$ is representable and let $(P;\tau ,\leq )$ be a Priestley space for some topology $\tau$.\\

We claim that, for $x\in P$, there is a sequence $(x_i:\ 0\leq i<\omega )$ such that either, 
for $0\leq j<i<\omega$, $x_i<x_j$ and $x$ is the greatest lower bound of 
$\{ x_i\mid 0\leq i<\omega\}$ or,
for $0\leq i<j<\omega$, $x_i<x_j$ and $x$ is the least upper bound of 
$\{ x_i\mid 0\leq i<\omega\}$.\\

To justify the claim, we consider the various possibilities.  
If $x = p$, then setting $x_i = p_i$ yields, 
for $0\leq j<i<\omega$, $p<p_i<p_j$.  Moreover, 
for $y\in P\setminus P(0)$, $[y)\cap P(0)$ is finite.  
In particular, $p$ is the greatest lower bound of $\{ p_i\mid 0\leq i<\omega \}$.  
Similarly, for $x = p_{i_0,\ldots ,i_n}$, let $x_i = p_{i_0,\ldots ,i_n,i}$ 
for $0\leq i<\omega$.  If $n$ is even, then, for $0\leq i<j<\omega$, 
$$p_{i_0,\ldots ,i_n,i} < p_{i_0,\ldots ,i_n,j} < p_{i_0,\ldots ,i_n}.$$  
Since $p_{i_0,\ldots ,i_n}$ is the greatest lower bound of 
$[p_{i_0,\ldots ,i_n,i})$ and, for $y\in P\setminus P(n+1)$, 
$(y]\cap P(n+1)$ is finite, it follows that 
$p_{i_0,\ldots ,i_n}$ is the least upper bound of 
$\{ p_{i_0,\ldots ,i_n,i}\mid 0\leq i<\omega \}$.  
Likewise, if $n$ is odd, then, for $0\leq j<i<\omega$,
$$p_{i_0,\ldots ,i_n} < p_{i_0,\ldots ,i_n,i} < p_{i_0,\ldots ,i_n,j}.$$
Since $p_{i_0,\ldots ,i_n}$ is the least upper bound of 
$(p_{i_0,\ldots ,i_n}]$ and, for every 
$y\in P\setminus P(n+1)$, $[y)\cap P(n+1)$ is finite, it follows 
that $p_{i_0,\ldots ,i_n}$ is the greatest lower bound of 
$\{ p_{i_0,\ldots ,i_n,i}\mid 0\leq i<\omega \}$.\\

Using the above claim, we now show that every $x\in P$ is an 
accumulation point.\\

To see this, say $x$ is the greatest lower bound of 
$\{ x_i\mid 0\leq i<\omega \}$ where, for $0\leq j<i<\omega$, $x_i<x_j$.
For $0\leq i<\omega$, there exists a clopen increasing set $V_i$ 
such that $x_i\in V_i$ and $x_{i+1}\not\in V_i$.  Clearly, 
$\{ V_i\mid 0\leq i<\omega \}$ is an open cover of 
$S = \{ x_i\mid 0\leq i<\omega \}$ with no finite subcover.  
In particular, $S$ is not closed.  Choose $y\in cl(S)\setminus S$.  
If $y\not\geq x$, then there is a clopen decreasing set $U$ with $y\in U$ 
and $x\not\in U$, from which it follows that $U\cap S = \emptyset$, 
contradicting $y\in cl(S)$.  If $y>x$, then $y$ is not a lower bound of $S$, 
as $x$ is the greatest.  In particular, for some $0\leq n<\omega$, $x_n\not\geq y$.  
It follows that there is a clopen decreasing set $U$ with $x_n\in U$ and $y\not\in U$.  
Thus, $S\subseteq \{ x_0,\ldots ,x_n\}\cup U$, which is a closed set.  
On the other hand, $y\in P\setminus (\{ x_0,\ldots ,x_n\}\cup U)$, contradicting 
the fact that $y\in cl(S)$.  We conclude that $y = x$ and, in particular, 
that, as claimed, $x$ is an accumulation point.  As similar argument holds 
in the case that $x$ is the least upper bound of 
$\{ x_i\mid 0\leq i<\omega \}$ where, for $0\leq i<j<\omega$, $x_i<x_j$.\\

Suppose then that $L$ is a bounded distributive lattice such that 
$(D(L);\tau (L),\subseteq )$ (recall the notation introduced 
in \S\ref{A}) is homeomorphic and 
order-isomorphic to $(P;\tau ,\leq )$.  For $a$, $b\in L$, there 
correspond clopen decreasing sets $A$, $B$, respectively.  Suppose $a<b$.  
Then $A\subset B$ and it is possible to choose $x\in B\setminus A$.  
Since $x$ is an accumulation point, there exists a distinct element 
$y\in B\setminus A$.  Say, without loss of generality, $x\not\geq y$.  
Then there exists a clopen decreasing set $U$ with $x\in U$ and $y\not\in U$.  
Set $C = A\cup (B\cap U)$.  Then $C$ is a clopen decreasing set such that 
$A\subset C\subset B$.  In particular, $C$ corresponds to an element 
$c\in L$ such that $a<c<b$.  We conclude that $(Q;\leq )$ the rational interval 
$(0,1)$ is embeddable in $L$, that is, $(Q^+;\leq )$ the rational interval 
$[0,1]$ is a $(0,1)$-sublattice of $L$.  If one such embedding is denoted by 
$f^+:Q^+\longrightarrow L$, then $f$ corresponds to continuous 
order-preserving map $D(f):D(L)\longrightarrow D(Q^+)$ which is also onto.  
That is, there is a mapping from $P$ onto $D(Q^+)$.  Since $D(Q^+)$ is 
uncountable and $P$ is countable, this is impossible and, as required, 
we conclude that $(P;\leq )$ is not representable.
\end{proof}

\vspace*{1cm}

 \begin{tabbing}
State University of New York put it more to t

\= some more some more some some more \kill

M.~E.~Adams \> D.~van der Zypen \\
Department of Mathematics \> Allianz Suisse Insurance Company \\
State University of New York \> Bleicherweg 19 \\
New Paltz, NY 12561, U.S.A.\> CH-8022 Zurich, Switzerland \\
adamsm@newpaltz.edu \> dominic.zypen@gmail.com\\

\end{tabbing}

\end{document}